\documentclass[letterpaper, 10 pt, conference]{ieeeconf}
\IEEEoverridecommandlockouts
\usepackage[noadjust]{cite}
\usepackage{amsmath,amssymb,amsfonts}
\usepackage{algorithmic}
\usepackage{graphicx}
\usepackage{textcomp}
\usepackage{xcolor}
\usepackage{dsfont}
\usepackage{amsthm}

\let \labelindent \relax
\usepackage{enumitem}
\usepackage{subcaption}
\usepackage[font=small]{caption}

\usepackage{accents}

\newtheorem{theorem}{Theorem}[section]
\newtheorem{corollary}{Corollary}[theorem]
\newtheorem{lemma}[theorem]{Lemma}
\newtheorem{remark}{Remark}[section]
\newtheorem{proposition}[theorem]{Proposition}
\newtheorem{assumption}{Assumption}

\newcommand{\R}{{\mathds{R}}}

\newcommand{\edits}[1]{{\color{black}{#1}}}

\DeclareMathOperator{\sign}{sign}

\def\BibTeX{{\rm B\kern-.05em{\sc i\kern-.025em b}\kern-.08em
    T\kern-.1667em\lower.7ex\hbox{E}\kern-.125emX}}
\begin{document}

\title{
Switching transformations for \edits{decentralized} control of opinion patterns in signed networks:
application to dynamic task allocation
\thanks{Supported by ONR grant N00014-19-1-2556, ARO grant W911NF-18-1-0325, DGAPA-UNAM PAPIIT grant IN102420, and Conacyt grant A1-S-10610, and by  
NSF Graduate Research Fellowship  
DGE-2039656.}
}

\author{Anastasia~Bizyaeva, Giovanna~Amorim, Mar\'ia~Santos,
        Alessio~Franci,
        and~Naomi~Ehrich~Leonard
\thanks{A. Bizyaeva, G. Amorim, M. Santos and N.E. Leonard are with the Department
of Mechanical and Aerospace Engineering, Princeton University, Princeton,
NJ, 08544 USA; e-mail: \{bizyaeva, giamorim, maria.santos, naomi\}@princeton.edu.}
\thanks{A. Franci is with the Department of Mathematics, National Autonomous University of Mexico, 04510 Mexico City, Mexico. e-mail: afranci@ciencias.unam.mx}}

\maketitle

\begin{abstract}
We propose a new \edits{decentralized} design to control opinion patterns on signed networks of agents making decisions about two options and to switch the network from any opinion pattern to a new desired one. Our method relies on \textit{switching transformations}, which switch the sign of an agent's opinion at a stable equilibrium by flipping the sign of the interactions with its neighbors. The global dynamical behavior of the switched network can be predicted rigorously when the original, and thus the switched, networks are structurally balanced. Structural balance ensures that the network dynamics are monotone, which makes the study of the basin of attraction of the various opinion patterns amenable to monotone systems theory. We illustrate the utility of the approach through scenarios motivated by multi-robot coordination and dynamic task allocation.
\end{abstract}

\section{Introduction}
\label{sec:intro}
Modern networked technologies require decentralized mechanisms for decision-making and allocation of tasks. 
For example, systems such as smart power grids, cloud computing services, or multi-robot teams, call for strategies that dynamically distribute tasks among individual units \edits{to optimize system performance even as task requirements change or units experience failure}.  

We use the model of networked nonlinear opinion dynamics of \cite{bizyaeva2022,Gray2018} to illustrate how network interconnection topology can be designed so  a group of decision makers converges to a desired opinion pattern and how the network can be transformed so  the group switches to a desired alternative opinion pattern.
\edits{When all agents commit to the same option the network is in `agreement', while for any other opinion configuration it is in `disagreement'. The emergence of agreement and disagreement in nonlinear opinion networks has been studied in \cite{bizyaeva2022,BizyaevaACC2020,Bizyaeva2021}. However, the analysis in those works has assumed that all network interactions are either positive or negative. In this paper we add to this body of analysis by allowing mixed-sign interactions.} 


Decision-making with signed interactions has been studied on linear networks with averaging dynamics \cite{Altafini2013,Liu2017}, as well as with nonlinear consensus models \cite{fontan2017multiequilibria, fontan2021role} and biased assimilation models \cite{wang2020biased}. {\em The novelty of our approach is to use signed interactions on a network as a design tool}. Our design methodology drives a distributed system to a desired network state and allows any individual agent to respond to local contextual changes  and adjust its allocation by dynamically adjusting the sign of interaction with its neighbors. \edits{Since the strategy relies only on pairwise interactions between neighboring agents, it is decentralized and agnostic to the global topology of the network communication graph. } 

Our contributions are as follows. First, we prove that a network system  can be easily and intuitively controlled to any agreement or disagreement opinion pattern using standard tools from signed graph theory grounded in switching transformations of graphs. Second, we prove a sufficient condition for the networked state to converge to one of two available equilibrium configurations. Third, we show  how a pattern of equilibrium opinions can be changed dynamically through local updates of the network weights that follow the structure of a switching transformation. Fourth, we validate the theory 
\edits{with simulation examples.}


In Section II we introduce notation. Section III describes the opinion dynamics model and summarizes some of its properties. In Section IV we present new analysis of the model on signed graphs and propose a systematic design approach \edits{for agent allocation} across two tasks. In Section V we describe the asymptotic dynamics of trajectories on structurally balanced graphs. Section VI relates the features of the approach in the context of multi-robot task allocation. Final remarks are included in Section VII.

\section{Notation and Mathematical Preliminaries}
For any vectors $\mathbf{x} = (x_1, \dots, x_N) \in \mathds{R}^N$, $\mathbf{y} = (y_1, \dots, y_N) \in \mathds{R}^N$,  the standard Euclidean inner product is $\langle \mathbf{x},\mathbf{y} \rangle = \sum_{i = 1}^N x_i y_i$. Let $\mathbf{x} \succeq \mathbf{y}$ if $x_i \geq y_i$ for all $i = 1, \dots, N$, and  $\mathbf{x} \succ \mathbf{y}$ if $x_i > y_i$ for all $i = 1, \dots, N$. Define the operation $\odot$ as the element-wise product of vectors, $\mathbf{x} \odot \mathbf{y} = (x_1 y_1, \dots, x_N y_N)$.  $\mathbf{0}_N$  denotes the vector with all zero entries in $\mathds{R}^N$, and $\mathcal{I}_N$  the identity matrix in $\mathds{R}^{N \times N}$.

We study networks of $N$ agents with a signed communication graph $\mathcal{G} := (\mathcal{V}, \mathcal{E}, \sigma)$ where $\mathcal{V} = \{1, \dots, N \}$ is the \textit{vertex set}, $\mathcal{E}$ is the \textit{edge set}, and $\sigma: \mathcal{E} \to \{1,-1\}$ is a \textit{sign function} or \textit{signature} of the graph $\mathcal{G}$. We use the sensing convention such that $e_{ik} \in \mathcal{E}$ denotes a directed edge in $\mathcal{G}$ that points from vertex $i$ to vertex $k$, indicating that $k$ is a neighbor of $i$. We assume that the unsigned directed graph $\Gamma = (\mathcal{V}, \mathcal{E})$ underlying $\mathcal{G}$ is \textit{simple}, i.e. contains no self-loops $e_{ii} \not\in \mathcal{E}$ for all $i \in \mathcal{V}$, and there is at most one edge $e_{ik}$ in $\mathcal{E}$ that begins at vertex $i$ and ends at vertex $k$ for all $i,k \in \mathcal{V}$. We say that the graph $\mathcal{G}$ is \textit{strongly connected} if the edges contained in $\mathcal{E}$ form a path between any two nodes. 

Define $A = (a_{ik})$ to be the $N \times N$ \textit{signed adjacency matrix} of 
$\mathcal{G}$ whose entries $a_{ik} \in \{0,1,-1\}$ satisfy $a_{ik} = 0$ if $e_{ik} \not \in \mathcal{E}$ and $a_{ik} = \sigma(e_{ik})$ if $e_{ik} \in \mathcal{E}$. 
We use the symbol $\lambda^*$ to distinguish, when it exists, the real and unique eigenvalue of $A$ that satisfies $\operatorname{Re}(\lambda^*) > \operatorname{Re}(\lambda_i)$ for all eigenvalues $\lambda_i \neq \lambda^*$ of $A$. We denote the right and left eigenvectors of $A$ corresponding to $\lambda^*$ as $\mathbf{v}^*$ and $\mathbf{w}^*$, respectively. We always assume $\mathbf{v}^*, \mathbf{w}^*$ are normalized to satisfy $\langle \mathbf{w}^*,\mathbf{v}^*\rangle = 1$. 
We adapt the statement of the standard Perron-Frobenius theorem, e.g. as presented in \cite{strang2016}, to specialize to adjacency matrices of graphs with an all-positive signature. 
 
\begin{proposition}[Perron-Frobenius] \label{prop:PF}
Suppose $\sigma(e_{ik}) = 1$ for all $e_{ik} \in \mathcal{E}$ for some strongly connected graph $\mathcal{G}$. Then the following hold: 1) $\lambda^*$ exists; 2) $\lambda^* > 0$; 3) we can choose $\mathbf{v}^*,\mathbf{w}^*$ to satisfy $\mathbf{v}^* \succ \mathbf{0}_N$ and $\mathbf{w}^* \succ \mathbf{0}_N$.
\end{proposition}


\section{Nonlinear Opinion Dynamics Model}
The evolution of the opinion of $N$ agents on a signed network
choosing between two options is modeled in this paper according to the continuous-time multi-agent, multi-option nonlinear opinion dynamics model in \cite{bizyaeva2022}.

Let $x_{i} \in \R$ denote the opinion of agent $i$, where the magnitude of $x_{i}$ determines the agent's commitment to one of the two options such that a stronger (weaker) commitment to an option corresponds to a larger (smaller) $|x_{i}|$. If $x_{i} = 0$, the agent is said to be unopinionated, and, if $x_{i} > 0\ (< 0)$, agent $i$ prefers option 1 (option 2). 
We define the opinion state of the network as $\mathbf{x} = (x_{1}, ..., x_{N}) \in \R^{N}$, with $\mathbf{x} = \mathbf{0}_N$ being the neutral state of the group. The network is in an agreement state when sign($x_{i}$) = sign($x_{k}$) for all $i$, $k \in \{1, ..., N\}$ (i.e., if all the agents commit to the same option), and in a disagreement state \edits{otherwise}.

\edits{The evolution of agent $i$'s opinion is determined by a linear damping term and a saturated network interaction term}
\begin{equation}
\dot{x}_{i} = - d\,x_{i} +  u_{i} S\left(\alpha x_{i} + \gamma \textstyle\sum_{\substack{k=1 \\ k \neq i}}^N {a}_{ik} x_{k}\right),
\label{eq:opinion_dynamics_hom}
\end{equation}
\edits{where $d>0$ is the damping coefficient, $u_i>0$ regulates the relative strength of the two terms, and
the odd saturating function $S : \R \rightarrow \R$  acts on network interactions. Furthermore, $S$ satisfies} $S(0) = 0 $, $S'(0) = 1$, $\sign(S''(x)) = -\sign(x)$.\footnote{\edits{The presence of non-smooth (piece-wise linear) saturation functions can be tackled using methods from non-smooth analysis~\cite{leenaerts1998piecewise} and recent bifurcation-theoretical tools for linear complementarity systems~\cite{miranda2021equivalence}.}} Network interactions comprise self-reinforcement interactions, weighted by $\alpha \geq 0$, and neighbor interactions, weighted by $\gamma > 0$. \edits{The sign of network interactions is determined by the signed adjacency weight $a_{ik} \in \{0,1,-1\}$. Agent} $i$ cooperates (competes) with agent $k$ when 
$a_{ik} = 1 \ (= -1)$ and is indifferent to agent $k$'s opinion when $a_{ik} = 0$. \edits{In vector form, dynamics \eqref{eq:opinion_dynamics_hom} are
\begin{equation} \label{eq:op-dyn-matrix}
  \dot{\mathbf{x}} =   - d \, \mathbf{x} + U S(\alpha \mathbf{x} + \gamma A \mathbf{x}),
\end{equation}
where  $S(\mathbf{y}) := (S(y_1), \dots, S(y_N))$ for any  $\mathbf{y} \in \mathds{R}^N$ and $U = \operatorname{diag}(u_1, \dots, u_N)$.} 


\subsection{Network opinion formation through bifurcation}

The following proposition, adapted from \cite[Theorem IV.1]{bizyaeva2022} and \cite[Theorem IV.1]{Bizyaeva2021} and stated without proof, summarizes several key features of the opinion dynamics \eqref{eq:opinion_dynamics_hom}. 

\begin{proposition}[Opinion formation as a pitchfork bifurcation]\label{prop:pitchfork} Consider \eqref{eq:opinion_dynamics_hom} on a graph $\mathcal{G}$ with $u_i = u \geq 0$, \edits{$\alpha \geq 0$, $\gamma > 0$, $d>0$} for all $i = 1, \dots, N$, and assume a simple, real largest eigenvalue $\lambda^*$ exists. Suppose $ \alpha + \gamma \lambda^* > 0$ and $\langle \mathbf{w}^*, (\mathbf{v}^*)^3 \rangle > 0$, where ($\mathbf{v}^*)^3 = \mathbf{v}^* \odot \mathbf{v}^* \odot \mathbf{v}^*$. Then 1) for $0 \leq u < u^* := d/(\alpha + \gamma \lambda^*)$, the neutral equilibrium $\mathbf{x} = \mathbf{0}_N$ is locally exponentially stable; 2) for $u > u^*$,  $\mathbf{x} = \mathbf{0}_N$ is unstable, and two branches of locally exponentially stable equilibria $\mathbf{x}=\mathbf{x}_1^*$, $\mathbf{x}_2^*$ branch off from $(\mathbf{x},u) = (\mathbf{0}_N,u^*)$ along a manifold tangent at $\mathbf{x}=\mathbf{0}_N$ to $\operatorname{span}(\mathbf{v}^*)$. The two nonzero equilibria differ by a sign, i.e. $\mathbf{x}_2^* = -\mathbf{x}_1^*$.
\end{proposition}

Proposition \ref{prop:pitchfork} \edits{shows that} a group of decision-makers with opinion dynamics \eqref{eq:opinion_dynamics_hom} \edits{can} break deadlock and commit to an opinionated configuration when their level of attention $u_i$ is sufficiently large. \edits{Fig. \ref{fig:siwtch_graphic} provides a graphical illustration of the equilibria and their stabilty as a function of $u$ in the form of the pitchfork bifurcation described by the proposition.} Next we state other useful properties of \eqref{eq:opinion_dynamics_hom}.  

\begin{corollary}[Sufficient condition for agreement]
When $\mathcal{G}$ is strongly connected with an all-positive signature,  conditions of Proposition \ref{prop:pitchfork} are always satisfied. For $u > u^*$, one of the two new stable equilibria satisfies $\mathbf{x}^* \succ \mathbf{0}_N$. \label{cor:agreement}
\end{corollary}
\begin{proof}
The corollary follows from Proposition \ref{prop:pitchfork} and Proposition \ref{prop:PF}, since  $\lambda^*$ is the Perron-Frobenius eigenvalue, and eigenvectors $\mathbf{w}^*, \mathbf{v}^*$ have all-positive entries.
\end{proof}

We next show that the equilibria predicted by Proposition \ref{prop:pitchfork} are the only equilibria admitted by dynamics \eqref{eq:opinion_dynamics_hom} for a range of values of $u$, following similar arguments as those used for Laplacian-weighted nonlinear consensus networks in \cite{fontan2017multiequilibria}. \edits{We first state a necessary lemma.}

\begin{lemma}[Boundedness] \label{lem:boundedness}
Any compact set $\Omega_{r} \subset \mathds{R}^N$ of the form $ \Omega_r = \{ \mathbf{x} \in \mathds{R}^N \ \ s.t. \ \ | x_i | < r \max_j\{u_j\} / d, \;  \forall i,j \in \mathcal{V} \} $
with $r > 1$ is forward-invariant for \eqref{eq:opinion_dynamics_hom}.
\end{lemma}
\begin{proof}
The lemma follows directly from the more general result \cite[Theorem A.2]{bizyaeva2022}.
\end{proof}

\begin{corollary}[Uniqueness of Equilibria] \label{cor:uniqueness}
Suppose conditions of Proposition \ref{prop:pitchfork} are satisfied, and let $\lambda_{2}$ be an eigenvalue of $A$ satisfying $\operatorname{Re}(\lambda_2) \geq \operatorname{Re}(\lambda_i)$ for all eigenvalues $\lambda_i \neq \lambda^*$ of $A$. 1) $\mathbf{x} = \mathbf{0}_N$ is globally asymptotically stable on a forward-invariant compact set $\Omega \subset \mathds{R}^n$ containing the origin $\mathbf{x} = \mathbf{0}_N$, for all $u \in [0,u^*)$; 2) when $\operatorname{Re}(\lambda_2) \geq - \alpha/\gamma$, $u \in (u^*, u_2)$, the only equilibria the system admits are $\mathbf{0}_N$, $\mathbf{x}_1^*$, and $\mathbf{x}_2^*$, where $u_2 = d/(\alpha + \gamma \operatorname{Re}(\lambda_{2}))$; 3) when $\operatorname{Re}(\lambda_2) < -\alpha/\gamma$, the only equilibria the system admits in $\Omega$ for all $u > u^*$ are $\mathbf{0}_N$, $\mathbf{x}_1^*$, and $\mathbf{x}_2^*$.
\end{corollary}
\begin{proof}
1) Existence of $\Omega$\edits{, and thereby boundedness of solutions of \eqref{eq:opinion_dynamics_hom},} is established in Lemma \ref{lem:boundedness}.  Define $\Tilde{A} = \alpha \mathcal{I}_N + \gamma A$ with components $\Tilde{a}_{ij}$, and let $f_{i}(\mathbf{x}) = \sum_{j=1}^N \Tilde{a}_{ij} x_{j}$. Consider the continuously differentiable function 
    $V(\mathbf{x}) = \sum_{i = 1} \int_{0}^{f_{i}(\mathbf{x})} S(\eta) d \eta.$ 
\edits{Along trajectories of \eqref{eq:opinion_dynamics_hom}, $\dot{V}(\mathbf{x}) = S(\Tilde{A}\mathbf{x})^T \tilde{A} \dot{\mathbf{x}} = S(\Tilde{A}\mathbf{x})^T \tilde{A}(-d \, \mathbf{x} + u S(\tilde{A}\mathbf{x})) 
    = - d \, S(\Tilde{A}\mathbf{x})^T \Tilde{A} \mathbf{x} + u S(\Tilde{A}\mathbf{x})^T \Tilde{A} S(\Tilde{A}\mathbf{x})  \leq - S(\Tilde{A}\mathbf{x})^T (d \mathcal{I}_N - u \tilde{A}) S(\Tilde{A} \mathbf{x})$ (using $|S(y)| \leq |y|$ and $\operatorname{sign}(S(y)) = \operatorname{sign}(y)$). Since $d \mathcal{I}_N - u \tilde{A}$ is positive definite for $u \in [0,u^*)$, }
\begin{equation}
    \edits{\dot{V}(\mathbf{x}) \leq - (d - u (\alpha + \gamma \lambda^*)) S(\Tilde{A}\mathbf{x})^T S(\Tilde{A}\mathbf{x}) \leq 0.} \label{eq:lyap_fun_deriv}
\end{equation}
The set on which \eqref{eq:lyap_fun_deriv} is exactly zero is $\mathcal{N}(\tilde{A}) = \{ \mathbf{x} \in \mathds{R}^N \ \ s.t. \ \ \tilde{A} \mathbf{x} = \mathbf{0}_N \}$. By LaSalle's invariance principle \cite[Theorem 4.4]{Khalil2002} we conclude that the trajectories $\mathbf{x}(t)$ approach the largest invariant set in $\mathcal{N}(\tilde{A})$ as $t \to \infty$. If $\mathcal{N}(\tilde{A}) = \{ \mathbf{0}_N \}$, the corollary follows trivially. Let $\mathbf{x} \in \mathcal{N}(\tilde{A})$ and suppose $\mathbf{x} \neq \mathbf{0}_N $. Then $\dot{\mathbf{x}} = - d \mathbf{x}$, i.e. all trajectories that start in $\mathcal{N}(\tilde{A})$ decay to the origin exponentially in time, and the corollary follows.  Under the assumptions on $u$ stated in 2) and 3), the Jacobian matrix $J(\mathbf{x}) = - d \mathcal{I}_{N} + u \operatorname{diag}( S'(\tilde{A} \mathbf{x}) ) \tilde{A}$ is Hurwitz for all $\mathbf{x} \in \mathds{R}^N \setminus \{ \mathbf{0}_N \}$, \edits{the proof of which} follows closely the argument presented in \cite[Lemma 6]{fontan2021role} and we omit its details. 
\edits{By Proposition \ref{prop:pitchfork}}, for values of $u$ in a small neighborhood above $u^*$ exactly three equilibria exist. Since the Jacobian is \edits{nonsingular for all $\mathbf{x}\in \Omega$ and all $u \in (u^*, u_2)$}, by the implicit function theorem, the number of equilibria remains unchanged. \end{proof}
 


\section{Switching Transformation as a Design Tool for Synthesis of Opinion Patterns \label{sec:switch}}

When the communication graph $\mathcal{G}$ contains edges with a negative signature \edits{and its adjacency matrix $A$ has a simple leading eigenvalue}, the opinion-forming bifurcation of Proposition \ref{prop:pitchfork} results in disagreement network equilibria. We describe a simple synthesis procedure for generating a signed adjacency matrix that results in a desired pattern of opinions among the decision-makers following opinion dynamics \eqref{eq:opinion_dynamics_hom}. We first introduce a few important concepts from the theory of signed graphs; \edits{for more details on the theory  
we refer the reader to 
\cite{zaslavsky1982signed} 
and \cite{zaslavsky2013matrices}.} 


\subsection{Signed graphs and switching}

Let $\mathcal{W} \subset \mathcal{V}$ be a subset of nodes on a signed graph $\mathcal{G}$. \textit{Switching} a set $\mathcal{W}$ on the graph $\mathcal{G}$ refers to a mapping of the graph $\mathcal{G}$ to  $\mathcal{G}^{\mathcal{W}} = (\mathcal{V},\mathcal{E}, \sigma_{\mathcal{W}})$ where the signature of all the edges in $\mathcal{E}$ between nodes in $\mathcal{W}$ and nodes in its complement $\mathcal{V}\setminus \mathcal{W}$ reverses sign. We introduce the \textit{switching function} $\theta: \mathcal{V} \to \{1, -1\}$, where for any $i \in \mathcal{V}$, $\theta(i) = -1$ if $i \in \mathcal{W}$ and $\theta(i) = 1$ otherwise. Then the signature of the switched graph $\mathcal{G}^{\mathcal{W}}$ is generated as  
\begin{equation}
    \sigma_{\mathcal{W}}(e_{ik}) = \theta(i) \sigma(e_{ik}) \theta(k) \label{eq:switch}
\end{equation}
for all $e_{ik} \in  \mathcal{E}$. \edits{From \eqref{eq:switch} we see that the signature update for an edge between agents $i$ and $k$ depends only on their membership in the switching set $\mathcal{W}$. Thus, the edges between $i$ and $k$ flip sign if and only if exactly one of $i,k$ is in the switching set $\mathcal{W}$, and does not change sign if $i,k$ are both in $\mathcal{W}$ or in $\mathcal{V} \setminus \mathcal{W}$. } Importantly, switching a set $\mathcal{W}$ all at once generates the same graph $\mathcal{G}^{\mathcal{W}}$ as sequentially switching individual vertices in $\mathcal{W}$. If $\mathcal{G}$ can be transformed into $\mathcal{G}^{\mathcal{W}}$ by switching,  $\mathcal{G}$ and $\mathcal{G}^{\mathcal{W}}$ are \textit{switching equivalent graphs}. 


Let $\theta$ be the function for switching from  
graph $\mathcal{G}$ to $\mathcal{G}^{\mathcal{W}}$, with adjacency matrices  $A$ and $A^{\mathcal{W}}$, respectively. Define the \textit{switching matrix} $\Theta = \operatorname{diag}(\theta(1), \theta(2), \dots, \theta(N))$. The adjacency matrices of $\mathcal{G}$ and its switching $\mathcal{G}^{\mathcal{W}}$ are related as
\begin{equation}
    A^{\mathcal{W}} = \Theta^{-1} A \Theta. \label{eq:switching_transf}
\end{equation}
\edits{Since $\Theta$  is diagonal and $\theta(i) = \pm 1$}, $\Theta^{-1}=\Theta$. We refer to \eqref{eq:switching_transf} as a \textit{switching transformation} of the adjacency matrix $A$, and  $A$ and $A^{\mathcal{W}}$ as \textit{switching equivalent adjacency matrices}. 
\begin{proposition}\label{prop:switch_spectrum}
Suppose $\mathcal{G}$, $\mathcal{G}^{\mathcal{W}}$ are switching equivalent with adjacency matrices $A$ and $A^{\mathcal{W}}$ and  associated switching matrix $\Theta$. Then 1) $A$ and $A^{\mathcal{W}}$ are \edits{iso}spectral, i.e. have the same set of eigenvalues; 2) $\mathbf{v}$ ($\mathbf{w})$ is a right (left) eigenvector of $A$ corresponding to  eigenvalue $\lambda$ if and only if $\Theta \mathbf{v}$ ($\Theta \mathbf{w}$) is a right (left) eigenvector of $A^{\mathcal{W}}$ with the same eigenvalue.  
\end{proposition}
\begin{proof}
The proposition follows from the standard properties of a matrix similarity transformation, since  $A$ and $A^{\mathcal{W}}$ are related through a similarity transformation \eqref{eq:switching_transf}. 
\end{proof}
Proposition \ref{prop:switch_spectrum} implies that the eigenvectors of the switched adjacency matrix $A^{\mathcal{W}}$ are obtained from the eigenvectors of the original adjacency matrix $A$ by flipping the sign of each entry that corresponds to a node which is being switched. We will take advantage of this observation in our design of nonlinear opinion patterns on a network. 

\subsection{Nonlinear opinion patterns on switch equivalent graphs\label{sec:switch_patterns}}

In this section we show that a switching transformation of the nonlinear opinion dynamics \eqref{eq:opinion_dynamics_hom} is effectively a coordinate change, and two switching equivalent networks generate topologically equivalent flow and bifurcation diagrams. 

\begin{theorem}[Diffeomorphism between trajectories of switching equivalent systems]\label{thm:switch_eq}
Consider switching equivalent graphs $\mathcal{G}$, $\mathcal{G}^{\mathcal{W}}$ with adjacency matrices $A$ and $A^{\mathcal{W}}$ and with switching matrix $\Theta$. The trajectory $\mathbf{x}(t)$ is a solution to \eqref{eq:opinion_dynamics_hom} on $\mathcal{G}$ if and only if $\Theta \mathbf{x}(t)$ is a solution of \eqref{eq:opinion_dynamics_hom} on $\mathcal{G}^{\mathcal{W}}$.
\end{theorem}
\begin{proof}
Suppose $\mathbf{x}(t)$ is a solution of \eqref{eq:opinion_dynamics_hom} on $\mathcal{G}$. Multiplying both sides of \eqref{eq:op-dyn-matrix} by the switching matrix $\Theta$ yields 
\begin{multline*}
    \frac{d}{dt} (\Theta \mathbf{x}(t)) = \Theta\left(- d \, \mathbf{x}(t) + U S(\alpha \mathbf{x}(t) + \gamma A \mathbf{x}(t)) \right) \\
    = -d \, \Theta \mathbf{x}(t) + \Theta U S(\alpha \mathbf{x}(t) + \gamma A \mathbf{x}(t)) \\
    =   -d \, \Theta \mathbf{x}(t) + U S(\alpha \Theta \mathbf{x}(t) + \gamma A^{\mathcal{W}} \Theta \mathbf{x}(t)),
\end{multline*}
where the last step follows since $\Theta U = U \Theta$ and $-S(y) = S(-y)$.
This shows that $\Theta \mathbf{x}(t)$ is a solution of \eqref{eq:opinion_dynamics_hom} on $\mathcal{G}^{\mathcal{W}}$. The other direction follows by an identical proof. 
\end{proof}

\begin{corollary}[Switching a graph ``rotates" a pitchfork bifurcation] \label{cor:pitchfork_switch}
Consider \eqref{eq:opinion_dynamics_hom} with $u_i = u \geq 0$ for all $i = 1, \dots, N$ on the graphs described in Theorem \ref{thm:switch_eq}. Suppose $\mathcal{G}$ satisfies the conditions of Proposition \ref{prop:pitchfork}. Then $\mathcal{G}^{\mathcal{W}}$ also satisfies the conditions of Proposition \ref{prop:pitchfork}. Furthermore, $\mathbf{x}^*$ is an equilibrium on the bifurcation diagram on $\mathcal{G}$ at some $u$ if and only if $\Theta \mathbf{x}^*$ is an equilibrium on the bifurcation diagram of $\mathcal{G}^{\mathcal{W}}$ at the same $u$. 
\end{corollary}
\edits{
\begin{proof}
For \eqref{eq:opinion_dynamics_hom} on $\mathcal{G}^{\mathcal{W}}$, $\lambda^*$ is simple and $\alpha + \gamma \lambda^* > 0$ because $\mathcal{G}$, $\mathcal{G}^{\mathcal{W}}$ are isospectral. Additionally, $\langle \Theta \mathbf{w}^*, (\Theta \mathbf{v}^*)^3 \rangle = \sum_{i = 1}^N \theta(i)^4 w_i^* (v_i^*)^3 = \sum_{i=1}^Nw_i^* (v_i^*)^3 = \langle \mathbf{w}^*,\mathbf{v}^* \rangle > 0$ since $\theta(i) = \pm1$ for all $i \in \mathcal{V}$ and therefore the condition of Proposition \ref{prop:pitchfork} are satisfied. The rest of the corollary statement follows as a direct consequence of Theorem \ref{thm:switch_eq}. 
\end{proof}
}

\noindent We illustrate the intuition of Corollary \ref{cor:pitchfork_switch} in Fig. \ref{fig:siwtch_graphic}. 

\begin{theorem}[Switching complementary vertex sets generates the same flow]\label{thm:siwtch_complement}
Consider two switching equivalent graphs $\mathcal{G}^\mathcal{W}$, $\mathcal{G}^{\mathcal{V} \setminus \mathcal{W}}$, generated by switching a set of vertices $\mathcal{W}$ or its complement $\mathcal{V}\setminus \mathcal{W}$ on graph $\mathcal{G}$. Let the switching matrices in relation to $\mathcal{G}$ of these two graphs be $\Theta^{\mathcal{W}}$ and $\Theta^{\mathcal{V}\setminus\mathcal{W}}$ respectively. The trajectory $\mathbf{x}(t)$ is a solution of \eqref{eq:opinion_dynamics_hom} on $\mathcal{G}^{\mathcal{W}}$ if and only if it is also a solution of \eqref{eq:opinion_dynamics_hom} on $\mathcal{G}^{\mathcal{V} \setminus \mathcal{W}}$. 
\end{theorem}
\begin{proof}
Suppose $\mathbf{x}(t)$ is a solution of \eqref{eq:opinion_dynamics_hom} on $\mathcal{G}^{\mathcal{W}}$. Then by Theorem \ref{thm:switch_eq}, $\Theta^{\mathcal{W}} \mathbf{x}(t)$ is an equilibrium of \eqref{eq:opinion_dynamics_hom} on $\mathcal{G}$. \edits{Applying the complementary switching transformation, and observing that $\Theta^{\mathcal{V}\setminus\mathcal{W}} \Theta^{\mathcal{W}} = -\mathcal{I}_N$}, we see that $\Theta^{\mathcal{V}\setminus\mathcal{W}} \Theta^{\mathcal{W}} \mathbf{x}(t) = - \mathbf{x}(t)$ is a solution of \eqref{eq:opinion_dynamics_hom} on  $\mathcal{G}^{\mathcal{V} \setminus \mathcal{W}}$. By odd symmetry of the dynamic equations \eqref{eq:opinion_dynamics_hom}, $\mathbf{x}(t)$ is also a solution of \eqref{eq:opinion_dynamics_hom} on $\mathcal{G}^{\mathcal{V} \setminus \mathcal{W}}$. The proof {of the converse} follows the same steps in opposite order. 
\end{proof}

\begin{figure}
    \centering
    \includegraphics[width=\linewidth]{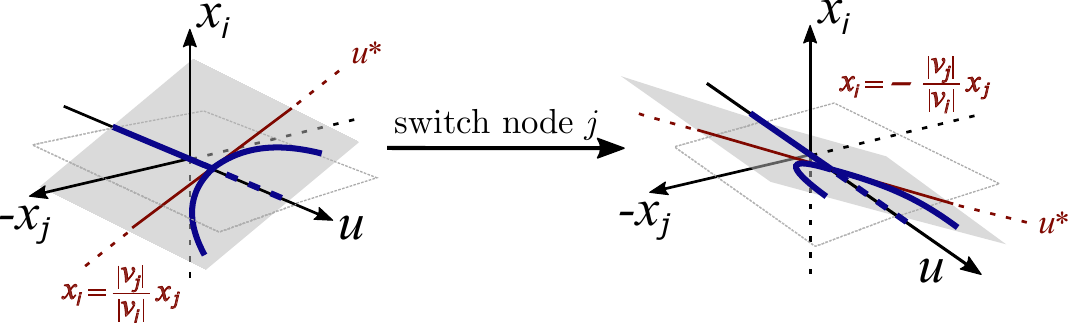}
    \caption{Illustration of Corollary \ref{cor:pitchfork_switch}. The bifurcation diagram of the switched system is a ``rotated" version of the original diagram because the sign of $v_j$ flips.}
    \label{fig:siwtch_graphic}
\end{figure}

\subsection{Synthesis of nonlinear opinion patterns}

The theoretical results of Section \ref{sec:switch_patterns}, 
lead to a design procedure to build a signed adjacency matrix that ensures a desired allocation of agents across the two options. \textit{\textbf{Step 1.} Start with a strongly connected $\mathcal{G}$ with an all-positive signature\edits{, i.e. $a_{ik} \in \{0,1\}$ for all $i,k \in \mathcal{V}$}}. By Corollary \ref{cor:agreement}, \eqref{eq:opinion_dynamics_hom} on $\mathcal{G}$ \edits{has an all-positive stable equilibrium $\mathbf{x}_1^*$ and all-negative stable equilibrium $\mathbf{x}_2^*$}. \textit{\textbf{Step 2.} \edits{
Define the switching set $\mathcal W$.}}
\edits{In this step, the designer chooses which nodes are grouped together. The two partitions correspond to the two tasks. \textit{\textbf{Step 3.} Update edge signatures of $\mathcal{G}$ locally as $a_{ik}^{\mathcal{W}} = \theta(i) a_{ik} \theta(k)$.  }} 
\edits{This edge signature update generates the switch-equivalent graph $G^\mathcal{W}$ and groups} all nodes in $\mathcal{W}$ and all nodes in $\mathcal{V}\setminus\mathcal{W}$ together by sign, i.e. the dynamics \eqref{eq:opinion_dynamics_hom} on $\mathcal{G}^\mathcal{W}$ \edits{is} bistable \edits{with stable} equilibria $\Theta \mathbf{x}_1^*$, $\Theta \mathbf{x}_2^*$.
If $|\mathcal{W}| = M$, the equilibrium $\Theta \mathbf{x}_1^*$ \edits{corresponds to} $M$ negative nodes, and  $\Theta \mathbf{x}_2^*$ \edits{to} $N-M$ negative nodes. \edits{We illustrate this in Fig. \ref{fig:applications_preassigned}.  Step 3 can also be implemented in a decentralized manner since it only relies on the pairwise switching states of neighboring agents}.


\begin{figure}
    \centering
    \includegraphics[width=1\linewidth]{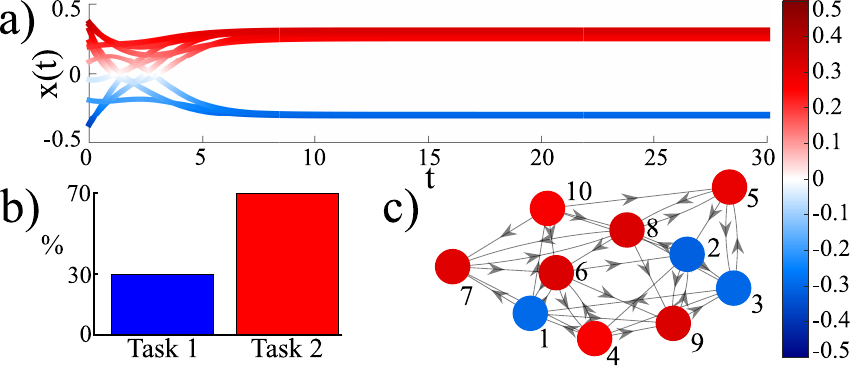}
    \caption{Assigning \edits{10 agents} to a 30-70\% distribution by switching agents 1, 2 and 3. (a) Time trajectory of the opinion dynamics. (b) Final agent distribution. (c) Network diagram with the opinion of each agent at $t$ = 30. Parameters: $S$ = tanh, $d$ = 1, $\alpha$ = 1.2, $\gamma$ = 1.3, $u$ = 0.324. Arrows on the graph defined by the sensing convention.}
    \label{fig:applications_preassigned}
\end{figure}


\section{Dynamic Switching \label{sec:dyn_switch}}

We next investigate the asymptotic opinion dynamics of  \eqref{eq:opinion_dynamics_hom} when the underlying communication graph $\mathcal{G}$ instantaneously changes to a switching equivalent graph $\mathcal{G}^{\mathcal{W}}$.

\subsection{Monotonicity and structural balance} 

First, we introduce some relevant definitions from the study of monotone systems. Let $\mathcal{K}$ be an orthant of $\mathds{R}^N$, $\mathcal{K} = \{\mathbf{x} \in \mathds{R}^N \ s.t. \ (-1)^{m_i} x_i \geq 0, \ i = 1, \dots, N \}$ with each $m_i \in \{0,1\}$. The orthant $\mathcal{K}$ generates a partial ordering ``$\leq_{\mathcal{K}}$'' on $\mathds{R}^N$ where if $\mathbf{x},\mathbf{y} \in \mathds{R}^N$, $\mathbf{y} \leq_{\mathcal{K}} \mathbf{x}$ if and only if $\mathbf{x} - \mathbf{y} \in \mathcal{K}$. We say a system $\dot{\mathbf{x}} = f(\mathbf{x})$ on $\mathcal{U} \subseteq \mathds{R}^N$ is \textit{type $\mathcal{K}$ monotone} if its flow preserves the partial ordering $\leq_{\mathcal{K}}$, i.e. if $\mathbf{x}_1(0) \leq_{\mathcal{K}} \mathbf{x}_2(0)$ implies $\mathbf{x}_1(t) \leq_{\mathcal{K}} \mathbf{x}_2(t)$ for all $t > 0$. 

\begin{lemma} \label{lem:balance}
Consider \eqref{eq:opinion_dynamics_hom} on a signed graph $\mathcal{G}$. It is a type $\mathcal{K}$ monotone system if and only if $\mathcal{G}$ is switching equivalent to $\mathcal{G}^+$, for which $\sigma(e_{ik}) = 1$ for all $e_{ik} \in \mathcal{E}$, i.e. $\mathcal{G}$ is \textbf{structurally balanced}.
\end{lemma}
\begin{proof}
The off-diagonal terms of the Jacobian matrix $J(\mathbf{x})$ are $u \gamma \operatorname{diag}(S'((\alpha \mathcal{I}_N + \gamma A) \mathbf{x})) A$. Let $\Theta$ be the switching matrix between $\mathcal{G}$ and $\mathcal{G}^+$. Since $S'(y) > 0$ for all $y \in \mathds{R}$, the matrix $u \gamma  \Theta \operatorname{diag}(S'((\alpha \mathcal{I}_N + \gamma A) \mathbf{x})) A \Theta$ has nonnegative components, and the lemma follows by \cite[Lemma 2.1]{smith1988systems}.
\end{proof}

\subsection{Instantaneous switching}


Suppose $\mathbf{x}^*$ is a hyperbolic equilibrium of \eqref{eq:opinion_dynamics_hom}, i.e. the linearization of the system at $\mathbf{x}^*$ has $m$ unstable eigenvalues and $N-m$ stable eigenvalues. Then by \cite[Theorem 1.3.2]{guckenheimer2013nonlinear}, there exist smooth local unstable and stable manifolds $W^u_{loc}(\mathbf{x}^*)$, $W^s_{loc}(\mathbf{x}^*)$ of dimensions $m$, $N-m$ that are tangent to the unstable and stable eigenspaces of the linearized systems at $\mathbf{x}^*$ and invariant under the dynamics. Global stable and unstable manifolds $W^s(\mathbf{x}^*)$,$W^u(\mathbf{x}^*)$ invariant under the dynamics can be obtained by continuing the trajectories in their local counterparts forwards or backwards in time.

\begin{assumption}[Stable manifold of origin is bounded; Fig.~\ref{fig:assumption_visual}]
Consider \eqref{eq:opinion_dynamics_hom} on some structurally balanced graph $\mathcal{G}$ with $u_i = u > u^*$ and $u<u_2$ when appropriate, as defined in Corollary \ref{cor:uniqueness}. Let \edits{$\mathcal{U}' \subset \mathds{R}^N$} be an open neighborhood containing the origin, and let $\mathbf{x} \in W^{s}(\mathbf{0}) \edits{\cap \mathcal{U}'}$. 1) $|\langle\mathbf{w}^*,\mathbf{x}\rangle| <  \varepsilon \| \mathbf{x} \|^2$ for some $0 < \varepsilon < 1$; 2) for equilibria $\mathbf{x}^*_k \neq \mathbf{0}$ of Proposition \ref{prop:pitchfork} with $k \in \{1,2\}$, $|\langle\mathbf{w}^*,\mathbf{x}^*_k\rangle| >  \varepsilon \| \mathbf{x}^*_k \|^2$. \label{asm:st_manifold}
\end{assumption}

\begin{figure}
    \centering
    \includegraphics[width=0.5\linewidth]{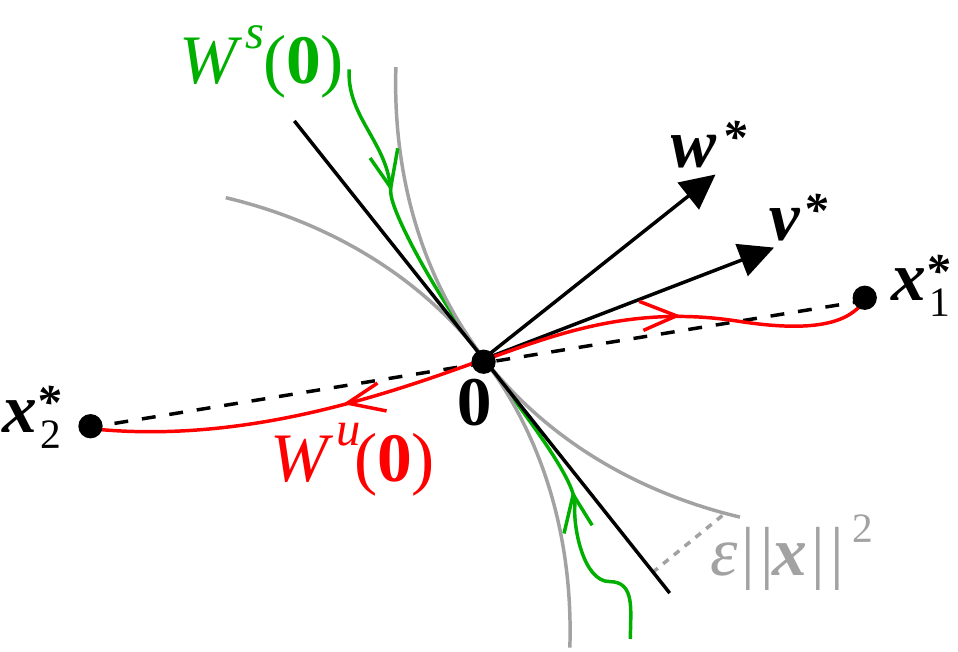}
    \caption{Geometric intuition behind Assumption \ref{asm:st_manifold}. The one-dimensional unstable manifold $W^u(\mathbf{0})$ of the origin (shown in red) forms heteroclinic orbits with the stable equilibria $\mathbf{x}_1^*,\mathbf{x}_2^*$, as is generically the case for monotone systems - see \cite[Theorem 2.8]{smith1988systems}.}
    \label{fig:assumption_visual}
\end{figure}

Estimating the $\varepsilon$ bound described above requires a lengthy computation of the stable manifold approximation \edits{(see \cite[p.132]{guckenheimer2013nonlinear} for an example of an invariant manifold approximation)} which we do not carry out for space considerations. However, this assumption should hold at least locally as a consequence of the (Un)Stable Manifold Theorem \cite[Theorem 1.3.2]{guckenheimer2013nonlinear} and monotonicity of the flow. We verified the assumption numerically for several graphs. \edits{For example, numerically we find $\varepsilon=0.05$ to be a valid bound for the graph and parameter values in Fig. \ref{fig:feature3} with $\mathbf{w}^*$ normalized to unit norm; in general $\varepsilon$ will vary with $u,d,\alpha,\gamma$}. 
\begin{lemma}[Regions of attraction]\label{lem:RoA}
Consider \eqref{eq:opinion_dynamics_hom} on some structurally balanced graph $\mathcal{G}$ with $u_i = u > u^*$ for all $i = 1, \dots, N$, on an open and bounded neighborhood $\Omega_r$ as defined in Lemma \ref{lem:boundedness}. Let $\mathbf{x}_1^*$, $\mathbf{x}_2^*$ be the nonzero equilibria described in Proposition \ref{prop:pitchfork} with $\langle \mathbf{w}^*, \mathbf{x}_1^* \rangle > 0$. Consider an initial condition $\mathbf{x}(0)$ at $t = 0$. 
If $\langle \mathbf{w}^*, \mathbf{x}(0) \rangle > \varepsilon \| \mathbf{x}(0) \|^2 (< -\varepsilon \| \mathbf{x}(0) \|^2)$ then as $t \to \infty$, $\mathbf{x}(t) \to \mathbf{x}_1^* (\mathbf{x}_2^*)$ .  
\end{lemma}
\begin{proof}
We established in Corollary \ref{cor:uniqueness} that the only equilibria the system admits are $\mathbf{0},\mathbf{x}_1^*,\mathbf{x}_2^*$, and $\Omega_r$ is positively invariant by Lemma \ref{lem:boundedness}. Let $B(\mathbf{x}_i^*)$ be the basin of attraction of equilibrium $\mathbf{x}_i$ in $\Omega_r$. By monotonicity (Lemma \ref{lem:balance}) and \cite[Theorem 2.6]{smith1988systems}, the set $\operatorname{Int}(B(\mathbf{x}_1^*)) \cup \operatorname{Int}(B(\mathbf{x}_2^*))$ is open and dense in $\Omega_r$, where $\operatorname{Int}$ signifies the interior points. Then following Assumption \ref{asm:st_manifold}, the stable manifold partitions $\Omega_r$ into the basins of attraction of the two locally asymptotically stable equilibria. The sets $U_{+} = \{\mathbf{x} \in \Omega_r \ s.t. \langle\mathbf{w}^*,\mathbf{x}\rangle > \varepsilon \| \mathbf{x} \|^2\}$, $U_{-} = \{\mathbf{x} \in \Omega_r \ s.t. \langle\mathbf{w}^*,\mathbf{x}\rangle < -\varepsilon \| \mathbf{x} \|^2\}$ do not intersect the center manifold and are therefore positively invariant. Then since $\mathbf{x}_1 \in U_{+}$ and $\mathbf{x}_{2} \in U_{-}$, we get that $U_{+} \subset B(\mathbf{x}_1^*)$ and $U_{-} \subset B(\mathbf{x}_2^*)$.
\end{proof}
\begin{remark}
In practice, without a precise value for the bound $\varepsilon$ from Assumption \ref{asm:st_manifold}, for most points $\mathbf{x}(0) \in \mathds{R}^N$ it is sufficient to check whether the projection of $\mathbf{x}(0)$ onto $\mathbf{w}^*$ is positive or negative to determine which region of attraction the points belongs to, i.e.  $\langle \mathbf{w}^*,\mathbf{x}(0) \rangle > 0 (< 0)\label{eq:convergence_condition}$
where $>$ implies convergence to $\mathbf{x}_1^*$ and $<$ to $\mathbf{x}_2^*$. This is because the stable manifold that partitions the space of possible opinion configurations occurs near the plane of points normal to $\mathbf{w}^*$ at the origin; see Fig. \ref{fig:assumption_visual} for illustration. As long as $\mathbf{x}(0)$ is not too close to this plane, the projection is a reliable heuristic for the asymptotic dynamics of the network opinions.
\end{remark}

\begin{theorem} \label{thm:dynamic_switch}
Consider \eqref{eq:opinion_dynamics_hom} on some $\mathcal{G}$ and let $\mathbf{x}_1^*$, $\mathbf{x}_2^*$  be the nonzero equilibria described in Proposition \ref{prop:pitchfork}, with $\langle \mathbf{w}^*,\mathbf{x}_1^*\rangle > 0$. Let $\mathcal{G}^{\mathcal{W}}$ be switch equivalent to $\mathcal{G}$ with the associated switching matrix $\Theta$. Suppose at $t = 0$,
\edits{$\mathbf{x}(0)$ is close to $\mathbf{x}_i^*$ where $i =1$ or $2$.} 
If $|\langle \Theta \mathbf{w}^*, \mathbf{x}_i^* \rangle| > \varepsilon \| \mathbf{x}_i^*\|^2$ and $\langle \Theta \mathbf{w}^*, \mathbf{x}_i \rangle > 0 (< 0)$ then for \eqref{eq:opinion_dynamics_hom} on $\mathcal{G}^{\mathcal{W}}$ as $t \to \infty$, $\mathbf{x}(t) \to \Theta \mathbf{x}_i (\to - \Theta \mathbf{x}_i)$. 
\end{theorem}
\begin{proof}
Without loss of generality,\edits{ let $\|\mathbf{x}(0) -\mathbf{x}_{1}^*\| < \mu$ }so that $\langle \Theta \mathbf{w}^*,\mathbf{x}(0)\rangle >  \varepsilon \| \mathbf{x}(0)\|^2$ and $w_{i}^* x_i(0) > 0 $ for all $i \in \mathcal{V}$ (these are true at $\mathbf{x}_1$ by assumption; sufficiently close nearby points will satisfy the conditions by continuity). By Theorem \ref{thm:switch_eq} we know that that for \eqref{eq:opinion_dynamics_hom} on $\mathcal{G}^{\mathcal{W}}$, $\Theta \mathbf{x}_1$ is an equilibrium, and the vector $\Theta \mathbf{w}^*$ is normal to the stable eigenspace at the origin. The theorem follows by Lemma \ref{lem:RoA}.
\end{proof}
Theorem \ref{thm:dynamic_switch} \edits{shows} that instantaneously changing a structurally balanced graph $\mathcal{G}$ to its switching equivalent $\mathcal{G}^{\mathcal{W}}$ \edits{results} in a predictable transition of the system state. Namely, if the number of nodes in $\mathcal{W}$ is small in comparison with the cardinality of $\mathcal{V}$, we expect that all nodes in $\mathcal{W}$ will change sign, and all of the nodes in $\mathcal{V}\setminus \mathcal{W}$ will not. A simulation example of this behavior is shown in Fig. \ref{fig:applications_oneswitch}. The precise number of nodes that can be switched simultaneously to generate this behavior depends on the eigenvector $\mathbf{w}^*$ of the graph adjacency matrix, the value of the equilibrium $\mathbf{x}_1^*$, and the bound $\varepsilon$. \edits{In practice,} it is often sufficient that $|\mathcal{W}| < \frac{1}{2} |\mathcal{V}|$. \edits{For the graph and parameter values in Fig. \ref{fig:feature3}, and the numerical estimate $\varepsilon=0.05$, any combination of 4 or fewer nodes can indeed be switched simultaneously. }

The analysis in this section \edits{reveals} that dynamics \eqref{eq:opinion_dynamics_hom} should be well-behaved if the transition between $\mathcal{G}$ and $\mathcal{G}^{\mathcal{W}}$ \edits{is driven by smooth dynamics, e.g., a suitably designed feedback law. We consider an example of such a smoothly-driven} transition in the following section.


\section{Applications to Multi-Robot Task Allocation}


We illustrate how our method can be \edits{used} to change the \edits{proportion} of robots dedicated to a task, \edits{how it can be decentralized, how it is robust to individual robot failures or additions,} 
and how we can ensure that robots  
\edits{switch  
when} triggered locally.

\paragraph{Task distributions} Many multi-robot applications need  subteams to be assigned to different tasks in a certain proportion (see e.g. \cite{Khamis2015} for a review on multi-robot task allocation). As illustrated in Fig.~\ref{fig:applications_preassigned}, switching transformations can be used to distribute a team of agents among two tasks in predetermined proportion.
\begin{figure}[t]
    \centering
    \includegraphics[width=1\linewidth]{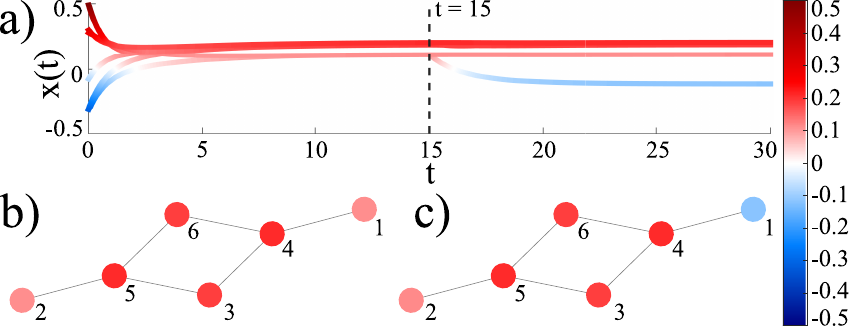}
    \caption{{Applying a switching transformation to agent 1 at t = 15. (a) Time trajectory of the opinion dynamics (b)-(c) Network diagram with the opinion of each agent at $t$ = 10, $t$ = 30. Parameters: $S$ = tanh, $d$ = 1, $\alpha$ = 1.2, $\gamma$ = 1.3, $u$ = 0.294. }}
    \label{fig:applications_oneswitch}
\end{figure}

In scenarios where robots can be divided in subteams, and each subteam can indiscriminately contribute to one of the tasks, our method guarantees a predetermined proportion of agents among tasks, but does not control which subteam is assigned to which task. For example, in Fig. \ref{fig:applications_preassigned}, the team composed of agents 1, 2, and 3 could execute either task 1 or 2. This affords flexibility from the application point of view, where the initial condition on the opinion (e.g., how close a robot is to an area where the demand for a task is abundant) can determine the final distribution of agents.



\paragraph{Local flexibility} Our approach provides flexibility by
letting single agents/robots make individual decisions \edits{in a decentralized way} without disconnecting from the network. \edits{This is important for multi-robot teams that should change their allocation across tasks in response to changing environmental conditions that are observed only by some agents; see \cite{ShinkyuICRA} for the case of globally available environmental cues in a multi-robot trash pick-up problem.} This method is also useful for long-duration autonomy  applications where  team performance should not be affected by failures or individuals that stop contributing to tasks, e.g. to charge batteries \cite{Notomista2021persistification}. See Fig. \ref{fig:applications_oneswitch}, where one agent switches between tasks without affecting the task preference of its neighbors in the graph.




To illustrate suppose that agent $i$ wants to switch options.  Agent $i$ alerts its neighbors that $\theta_i = -1$. 
We define
\begin{equation}
    \tau_{a} \dot{a}_{ik} = -a_{ik} + a_{ik}(0) \theta_{i} \theta_{k}, \label{eq:a_dynamics}
\end{equation}
where $\tau_a$ is a time-scale parameter and $a_{ik}(0) \in \{0,1,-1\}$ is the initial signature of the edge between agents $i$ and $k$. 
As seen in Fig. \ref{fig:feature3}, the dynamics \eqref{eq:a_dynamics} allow locally originated switches to take place simultaneously between neighbors. This property can be useful in applications involving switching cascades, where a robot switching to a new task can trigger its neighbors to switch. This feature is relevant in dynamic task allocation for multi-robot systems where robots can assign themselves to new tasks as a result of interacting with the environment or with their neighboring robots  \cite{Lerman2006analysis}.


\begin{figure}[t]
    \centering
    \includegraphics[width=1\linewidth]{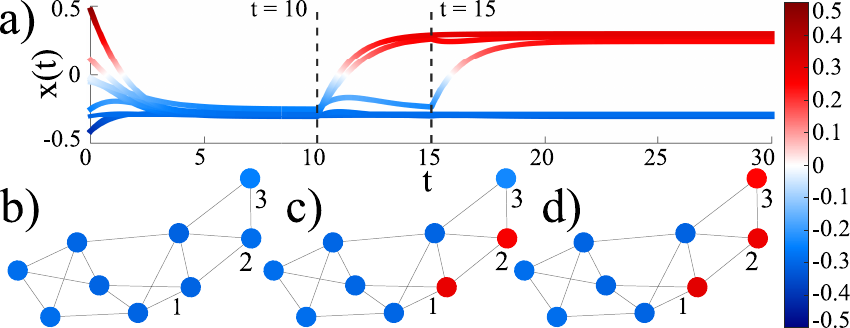}
    \caption{Locally switching agents 1, 2 and 3 
    using \eqref{eq:a_dynamics}. Agents 1, 2 switch at $t = 10$ and agent 3  at $t = 15$. (b)-(d) Network diagram with the opinion of each agent at $t$ = 10, $t$ = 15, $t$ = 30. Parameters: $S$ = tanh, $d$ = 1, $\alpha$ = 1.2, $\gamma$ = 1.3, $u$ = 0.315, $\tau_{a}$ = 0.01}
    \label{fig:feature3}
\end{figure}

\section{Final Remarks}
We analyzed the nonlinear networked opinion dynamics \eqref{eq:opinion_dynamics_hom} on signed graphs and proposed a novel approach for \edits{dynamic and decentralized allocation of} a group of agents across two tasks. In future work, we  aim to generalize the results in Section \ref{sec:dyn_switch} to graphs that are not structurally balanced, to derive an estimate for the $\varepsilon$ bound from Assumption \ref{asm:st_manifold}, and to extend this analysis to the more general multi-option opinion dynamics of  \cite{bizyaeva2022}. 



\bibliographystyle{./bibliography/IEEEtran}
\bibliography{./bibliography/references}

\end{document}